\documentclass[12pt]{amsart}
\usepackage{amssymb,amsmath}
\usepackage{geometry}    
\usepackage{mathrsfs}

\usepackage{graphicx}

\usepackage{vmargin}

\usepackage{tikz}  

\setmargrb{1in}{1in}{1in}{1in}

\newtheorem{theorem}{Theorem}[section]
\newtheorem{lemma}[theorem]{Lemma}

\newtheorem{corollary}[theorem]{Corollary}
\theoremstyle{definition}
\newtheorem{definition}{Definition}
\newtheorem{remark}{Remark}

\usepackage{xcolor}

\newcommand{\wo}{\widehat{w}_n(\xi) }

\begin{document}
	
	\title[Uniform approximation in small dimension]{Uniform dual approximation 
		to Veronese curves in small dimension}

\author{Johannes Schleischitz}

\thanks{Middle East Technical University, Northern Cyprus Campus, Kalkanli, G\"uzelyurt \\
	johannes@metu.edu.tr ; jschleischitz@outlook.com}

\begin{abstract}
We refine upper bounds for the classical exponents of uniform approximation for a linear form on the Veronese curve in dimension from $3$ to $9$.
For dimension three, this in particular 
shows that a bound previously 
obtained by two different methods is not sharp.
Our proof involves parametric geometry of numbers
and investigation of geometric properties of best approximation
polynomials. Slightly stronger bounds have been obtained by Poels
with a different method contemporarily. In fact, we obtain the same
bounds as a conditional result.
\end{abstract}

\maketitle

\section{Introduction}

\subsection{New results} \label{intro}
Davenport and Schmidt \cite{davsh}, in the course of investigating approximation
to real numbers by algebraic integers related to the famous open problem of Wirsing \cite{wirsing}, implicitly studied uniform
exponents of approximation on the Veronese curve in dimension $n$
defined as $\{ (\xi,\xi^2,\ldots,\xi^n): \xi\in \mathbb{R} \}$. 
Two variants of these
exponents were addressed in \cite{davsh}, one for simultaneous approximation to successive powers of $\xi$
and one for small values of a linear form (degree $n$ polynomial). 
Both types of exponents are indeed closely linked to
Wirsing's problem and variants of it, see besides \cite{davsh} 
for example also \cite{badsch}.
In this paper, we refine upper bounds for the uniform exponents 
with respect to the latter polynomial setting. 
For $\xi$ a real number and $n\ge 1$ an integer, let us denote them by $\widehat{w}_n(\xi)$ which are defined as the supremum of $w$ so that
\[
H(P)\le X, \qquad 0<\vert P(\xi)\vert < X^{-w}
\]
has a solution in an integer polynomial $P=P(X)$ 
of degree at most $n$,
for all large $X$.
Let us directly define the associated ordinary exponent of approximation $w_n(\xi)$
as well, given as supremum of $w$ so that
\[
0<\vert P(\xi)\vert < H(P)^{-w}
\]
has infinitely many solutions in integer polynomials $P$ of degree
at most $n$. Clearly $w_n(\xi) \ge \widehat{w}_n(\xi)$
by choosing $X=H(P)$, 
and moreover
\[
w_1(\xi)\le w_2(\xi)\le \cdots, \qquad
\widehat{w}_1(\xi)\le \widehat{w}_2(\xi)\le \cdots
\]
hold for any real number $\xi$. 
Moreover Dirichlet's Theorem shows the lower bounds
\[
w_n(\xi) \ge \widehat{w}_n(\xi)\ge n.
\]
A well-known consequence of the subspace theorem is that
for $\xi$ any real algebraic number of degree $d$ we have
$w_n(\xi)=\widehat{w}_n(\xi)=\min\{ n, d-1\}$, so we may only consider
transcendental numbers $\xi$ below.
It is well-known that $\widehat{w}_1(\xi)=1$ for all irrational
real numbers $\xi$, 
see Khintchine \cite{KH}, so we may restrict to $n\ge 2$.
As indicated above, upper bounds for $\widehat{w}_n(\xi)$ have first been studied by Davenport and Schmidt \cite{davsh}, whose result shows in our notation that
\begin{equation} \label{eq:2n-1}
\widehat{w}_n(\xi)\le 2n-1, \qquad n\ge 2,
\end{equation}
holds for any real $\xi$. For $n=2$, they proved a stronger bound of the form 
\begin{equation}  \label{eq:DSC}
    \widehat{w}_2(\xi)\le \frac{3+\sqrt{5}}{2}=2.6180\ldots,
\end{equation}
which 
surprisingly turned out to be optimal as shown by Roy \cite{royann}.
For $n\ge 3$ the optimal bound remains unknown. It took almost 50 years for the first small improvements to \eqref{eq:2n-1} in \cite{buschlei}, 
where an upper bound of order $2n-\frac{3}{2}+o(1)$  
as $n\to\infty$ with positive error term for each $n$ was established. The method also reproved the optimal upper bound \eqref{eq:DSC} for $n=2$. 
In fact, as noticed in \cite{2018}, 
the method in \cite{buschlei} when
combined with the later proved optimal ratio for ordinary and uniform
exponents by Marnat and Moshchevitin \cite{mamo}, directly yields
\begin{equation}  \label{eq:bsch}
\widehat{w}_n(\xi) \le \alpha_n:= \max\{ 2n-2 , \sigma_n \}=
\begin{cases}
\sigma_n,\qquad\qquad 2\le n\le 9,\\
2n-2, \qquad\; n\ge 10.
\end{cases}
\end{equation}
where $\sigma_n$ is the real solution to
\[
\frac{ (n-1)x }{x-n} - x +1 = \left(  \frac{n-1}{x-n}\right)^n
\]
in the interval $x\in (n,2n-1)$. 
The case $n=2$ indeed recovers \eqref{eq:DSC}.
We have $\sigma_n= 2n-C+o(1)$ as $n\to\infty$, where $C=2.25...$ is explicitly computable, see \cite{ichjp} for details. In case of strict inequality $w_n(\xi)> w_{n-1}(\xi)$, the term $2n-2$ in \eqref{eq:bsch} can be ignored, thereby implying the bound $\widehat{w}_n(\xi) \le \sigma_n$ which is stronger than \eqref{eq:bsch}
for $n\ge 10$.
The estimate $\widehat{w}_n^{\ast}(\xi)\le \sigma_n$ is true unconditionally, where $\widehat{w}_n^{\ast}(\xi)$ is a closely related classical exponent measuring uniform approximation by algebraic numbers of degree at most $n$, we prefer not to define it here and refer
to \cite{buschlei}. However the latter exponent is always bounded above by $\widehat{w}_n(\xi)$, hence giving a weaker claim.

In a later paper the author \cite{2018} introduced another method,
involving parametric geometry of numbers. It turned out that 
in \cite{2018}
the same bound for $n=3$ as in \eqref{eq:bsch} that reads $\widehat{w}_3(\xi)\leq 3+\sqrt{2}=4.4142\ldots=\alpha_3$
was obtained, however for larger $n$ the bounds became 
slightly weaker than $\alpha_n$.
Some significantly stronger but 
conditional results were stated in \cite[\S~2]{2018} as well.
In a very recent preprint that appeared just days before the current
paper, Poels~\cite{apoe} improved on \eqref{eq:bsch} obtained in \cite{buschlei} by establishing stronger
bounds of the form
\begin{equation} \label{eq:poelse}
\wo \le 2n-2, \quad (n\ge 4), \qquad \widehat{w}_3(\xi)\le 2+\sqrt{5}=4.23\ldots.
\end{equation}
Moreover, for large enough $n$, a bound of the form
\[
\wo\le 2n-\frac{1}{3}n^{1/3}
\]
was obtained in~\cite{apoe}.

In this paper, we refine the method from the latter paper \cite{2018} to improve the bound $\alpha_n=\sigma_n$ from \eqref{eq:bsch}
in the range $3\le n\le 9$. Unfortunately, our bounds in this range
will be weaker than the very recent estimates \eqref{eq:poelse}.
We want to point out however that
our method is considerably different from the one in~\cite{apoe}
and may be of independent interest for future improvements.
Our main result reads as follows.

\begin{theorem}  \label{H}
    For any $n\ge 2$ and any real number $\xi$ we have
      \[
    \widehat{w}_n(\xi) \leq \beta_n
    \]
    where $\beta_n$ is the root of the monic quartic polynomial
    \[
    Q_n(T)= T^4 + a_3 T^3 + a_2 T^2 + a_1T+ a_0 
    \]
    in the interval $(2n-2,2n-1)$, where 
      \begin{align*}
           a_3&= 4-4n, \\
    a_2&= 5n^2 - 12n + 8,\\
    a_1&= -2n^3 + 11n^2 - 18n + 7, \\  
    a_0&= - 2n^3 + 6n^2 - 4n.
    \end{align*}
%    $A_n(x)=B_n(x)$ for the 
 %   rational functions
 %   \[
%A_n:= \frac{ (n-1)(x-2n+3) x - (2n^2-5n+4)(x-2n+3) - n +2  }{ (n-1)(x-2n+3)%(1+x)  }
%\]
%and 
%\[
%B_n(x):= \frac{ (2n-1)\cdot ((3n-4-x)x+n(x-2n+3)) }{2(x-2n+3) (x-1)(x-%n)^2(1+x)}  - \frac{1}{2n-2}
%\]
\end{theorem}

Any polynomial $Q_n$ has four distinct real roots
and $\beta_n$ is the largest among them.
For $n=2$ we again obtain the optimal bound $\beta_2= 2.6180\ldots$ from \eqref{eq:DSC}
obtained by Davenport and Schmidt \cite{davsh} 
and independently in \cite{buschlei} and \cite{2018}.
For $n=3$ we get
\begin{equation}  \label{eq:FRIT}
\widehat{w}_3(\xi) \le \beta_3 = 4.3234\ldots < 4.4142\ldots = 3+\sqrt{2}= \alpha_3.
\end{equation}
The value $\alpha_3$ is from \eqref{eq:bsch} and was obtained 
with an alternative proof in \cite{2018} as well.
It turns out that similarly,
for $3\le n\le 9$ we get an improvement to \eqref{eq:bsch}.

We further provide new, stronger conditional results. For brevity we delay the definition
of best approximation polynomials and refer to~\S~\ref{pr1} below.

\begin{definition} \label{good}
	We say an integer $k\ge 2$ is {\em good} for $n,\xi$ if the triple of consecutive best approximation polynomials $\{ P_{k-1}, P_k, P_{k+1}\}$ defined
	in Definition~\ref{newdef} below associated to $n,\xi$ is linearly independent.
\end{definition}

It is well-known and follows from Lemma~\ref{le1} below that there
are infinitely many good $k$ for each pair $n,\xi$. Assuming refinements yield improvements as follows.

\begin{theorem}  \label{T2}
	Let $n\ge 2$ be an integer and $\xi$ be a real number. 
	\begin{itemize}
		\item[(i)] Assume that infinitely often, the integers $k-1$ and $k$ are both good for $n,\xi$.  
	Then
	\[
	\widehat{w}_n(\xi) \le \gamma_n
	\]
	where $\gamma_n$ is the root of the monic cubic polynomial
	\[
	R_n(T) := T^3 - (4n-4)T^2 + (5n^2-11n+6)T + (-2n^3+8n^2-10n+3).
	\]
	in the interval $(2n-2,2n-1)$.
	\item[(ii)] Assume all sufficiently large integers $k$ are good
	for $n, \xi$. Then
	\[
	\widehat{w}_n(\xi) \le \rho_n:= \max\left\{ \frac{ \sqrt{5}+1}{2}n - \frac{\sqrt{5}-1}{2}, 2n-2\right\}=\begin{cases}
	\frac{ \sqrt{5}+3}{2}, \qquad\quad n=2 \\
	\sqrt{5}+2, \qquad n=3 \\
	2n-2, \quad\quad\; n\ge 4.
	\end{cases}
	\]
		\end{itemize}
\end{theorem}

Note that the conditional bounds $\rho_n$ coincide with those
obtained unconditionally by Poels in \eqref{eq:poelse}. 

Again all roots of $R_n$ are real and $\gamma_n$ is the largest.
The condition in (i) can be relaxed to assuming that on a subsequence
of good $k$, we have 
\[
\frac{\log P_{\ell-1}}{\log P_k} \to 1, \qquad k\to\infty
\]
(understood inside the subsequence) with $\ell=\ell(k)\ge k+1$ chosen minimal so that $\{P_{k-1}, P_k, P_{\ell}\}$ are linearly independent. A similar
relaxation of the hypothesis in (ii) can be stated.
Conversely, a stronger assumption then in (i) is 
four consecutive best
approximation polynomials being linearly independent infinitely often.
In view of~\cite{nmg}, it is questionable if either condition 
is true for all real numbers $\xi$, see the comments in 
\S~\ref{on} below. Nevertheless,
we conjecture that at least the bounds $\gamma_n$ of Theorem~\ref{T2} hold unconditionally and it would be desirable to prove this.

The quality of the bounds are illustrated in
the following table,
with decimal expansions cut off after four digits.

\begin{center}
    \begin{tabular}{ |c|c|c|c|c| }
        \hline
        n & new bound $\beta_n$ & bound $\alpha_n=\sigma_n$ derived from  \cite{buschlei} & cond. bd. $\gamma_n$ & cond. bd. $\rho_n$ \\ \hline
        2 & 2.6180 & 2.6180 & 2.6180 & 2.6180 \\
        3 & 4.3234 & 4.4142 & 4.3028 & 4.2360 \\
        4 & 6.1592 & 6.2875 & 6.1451 & 6 \\
        5 & 8.0865 & 8.2010 & 8.0791 & 8 \\
        6 & 10.0528 & 10.1382 & 10.0488 & 10 \\
        7 & 12.0352 & 12.0906 & 10.0328 & 12  \\
        8 & 14.0251 & 14.0532 & 14.0236 & 14 \\
        9 & 16.0187 & 16.0231 & 16.0177 & 16 \\
        \hline
    \end{tabular}
\end{center}

%We further compute
%\begin{align*}
%\beta_4 = 6.1592\ldots,\quad
%\beta_5= 8.0865\ldots,\quad  \beta_6= 10.0528\ldots,
%\quad \beta_8=14.1088\ldots, \quad \beta_9= %16.0940\ldots.
%\end{align*}
%and
%\[
% \beta_7=12.0352\ldots, \quad \beta_8= %14.0251\ldots, \quad \beta_9= 16.0187\ldots. 
%\]
%These should be compared with the larger bounds %\eqref{eq:bsch} from [cite bug schlei] 
%numerically given as
%\begin{align*}
%    \alpha_4 = 6.2875\ldots,\quad
%\alpha_5= 8.2010\ldots,\quad  \alpha_6= 10.1382\ldots
%,\\ \quad \alpha_7&=12.1287\ldots,
%\quad \alpha_8=14.1088\ldots, \quad \alpha_9= 16.0940\ldots.
%\end{align*}
%and 
%\[
%\alpha_7= 12.0906\ldots, \quad \alpha_8= %14.0532\ldots, \quad \alpha_9= 16.0231\ldots.
%\]
As in \cite{2018}, as an artefact of the method, our bound satisfies the asymptotics
\[
\beta_n = 2n-2 +o(1), \qquad n\to\infty,
\]
with positive error terms for each $n$, however 
the limit being approached faster than in \cite{2018}. 
The same holds for $\gamma_n, \rho_n$.
Since $\alpha_n=2n-2$ for $n\ge 10$ by \eqref{eq:bsch}, our largest value $n=9$ where an improvement is obtained in either theorem above is a natural barrier, and surpassing it would probably require significant new ideas in the method. 
\\

We emphasize again that
especially the improvement compared to \eqref{eq:FRIT} for $n=3$ is remarkable, as it shows that the previously best known bound $3+\sqrt{2}$ that had been 
obtained from two different methods in \cite{buschlei, 2018}, is not sharp.
There is no reason to believe that the new bounds $\beta_n$
of Theorem~\ref{H} are optimal for any $n\ge 3$, possibly
the same is true for the conditional bounds.
Indeed, conversely, if $n\ge 3$ it is not known if
the exponent $\widehat{w}_n$ can take a value larger than $n$ for any
real number $\xi$. If this is not the case for some $n\ge 3$, 
this would imply an affirmative answer to Wirsing's problem (and some variants) for this
value of $n$, see~\cite[Lemma~1]{davsh}.

\subsection{Ideas of the proof and related exponents}
The foundation of our method will be similar to \cite{2018}, 
using an approach inspired by parametric geometry of numbers,
however with several technical twists. Firstly, we need a claim on consecutive best approximation polynomials (see Definition~\ref{newdef}) lying in 
two-dimensonal subspaces (Lemma \ref{le1}, only needed
for Theorem~\ref{H}).
Moreover, some step from the proof in \cite{2018} is simplified (via Lemma \ref{lemur}), which is necessary for our method here. Furthermore, some
technical improvement in the treatment of parametric geometry of numbers is obtained (Lemma \ref{le2}). These new elements will be combined in a concise way to beat the bounds from~\cite{2018}.

 While we restrict ourselves to the linear form exponents here,
we want to briefly remark on the according 
exponents usually denoted $\widehat{\lambda}_n(\xi)$ for the dual problem of simultaneous approximation to consecutive powers of a real number, also
closely connected to Wirsing's problem and variants. There similar progress 
in form of (rather small) improvements of the original bound by Davenport and Schmidt \cite{davsh} have been made, some very recently. The first improvement for odd $n$ was due to Laurent \cite{laurent}, and a stronger bound
for $n=3$ due to Roy \cite{roy3}. For even $n$, the author improved the bound from \cite{davsh} in \cite{equprin, period}. Then in a very recent paper Badziahin \cite{badz} improved on the previous results for $n\ge 4$, which in turn has been refined by Poels and Roy \cite{pr} to constitute the currently best known bounds for $\widehat{\lambda}_n(\xi)$.
In private communication, D. Roy pointed out to me that
he recently obtained a very small improvement on his result~\cite{roy3} for $n=3$, in a paper in preparation.
%A recent preprint by Marnat and Moshchevitin~\cite{mamo2}
%studies geometric aspects of the problem in the same case $n=3$,
%also aiming for a small improvement of\cite{roy3}.
However, as in the case of linear form
exponents, all obtained refinements compared to the original result by Davenport and Schmidt \cite{davsh} are rather small, 
moreover again it remains
unclear if the minimum value $1/n$ is exceeded for any $n\ge 3$ and any real number $\xi$. In summary,
both types of exponents remain rather poorly understood for $n\ge 3$.

\textbf{Acknowledgment}. The paper originates in the author's visit to the University of Sydney in the International
Visitor Program in 2023. The author thanks the University of Sydney
and his host Dzmitry Badziahin for the hospitality.

%Weaker version:  
%
%\begin{theorem} \label{th1}
%Assume $P_{k-1}, P_k, P_{k+1}$ are lin. indpt for all large $k$. Then 
%\[
%\widehat{w}_n(\xi) \le 
%\]
%\end{theorem}
%
% It is known that inf often three consecutive best apps are lin idpt, but not %necessarily all the time.

\section{On best approximation polynomials   } \label{prepa}

\subsection{Definition and a linear independence result} \label{pr1}

An important object for the study of exponents of Diophantine approximation are integer minimal points, used for example in~\cite{davsh67, davsh}, which are polynomials in our case.
Let us first define this sequence of best approximation polynomials $(P_k)_{k\ge 1}$ as in \cite[Definition 2.1]{2018}.

\begin{definition} \label{newdef}
For $n\geq 1$ an integer and $\xi$ a real number, an
integer polynomial $P$ of degree at most $n$ will be called
{\em best approximation polynomial associated to $(n,\xi)$}
if it minimizes $\vert P(\xi)\vert$ among all non identically 
zero integer polynomials of degree at most $n$ and height (maximum of absolute values of coefficients) at most $H(P)$. Any pair $n,\xi$ with $\xi$ not algebraic of degree at most $n$
thus
gives rise to a uniquely determined
(up to sign) infinite sequence 
of best approximation polynomials. We denote it by 
$(P_{k})_{k\geq 1}$ and the height
of $P_k$ by $H_k$.
\end{definition}

Recall that they satisfy
\begin{equation}  \label{eq:jopi}
\vert P_1(\xi)\vert > \vert P_2(\xi)\vert > \cdots, \qquad  
H_1 < H_2 < \cdots.
\end{equation}
We derive for $k\ge 2$ the quantities
\begin{equation}  \label{eq:nocce}
\mu_{k}:= -\frac{ \log |P_{k-1}(\xi)| }{\log H_k}, \qquad 
v_k:= -\frac{ \log |P_{k}(\xi)| }{\log H_k}.
\end{equation}
and see the relation to the classical exponents is given by
\begin{equation} \label{eq:conne}
\widehat{w}_n(\xi)= \liminf_{k\to\infty} \mu_k, \qquad
w_n(\xi)= \limsup_{k\to\infty} v_k.
\end{equation} 
We comprise some results that were implicitly derived in \cite{2018}, see more precisely the beginning
of the proof of \cite[Theorem 1.1]{2018} and \cite[\S 3.2]{2018}.
Following notation of \cite{2018},
let 
\begin{equation} \label{eq:Vk}
\mathscr{V}_k= \{ P_k , T P_k , \ldots, T^{n-2} P_k\}, \qquad k\ge 1,
\end{equation}
consisting of $n-1$ integer polynomials of degree at most $2n-2$.

\begin{lemma}[\cite{2018}]   \label{ichalt}
    Let $n\ge 2$ be an integer and $\xi$ a real number that satisfies
\begin{equation} \label{eq:2n}
    \widehat{w}_n(\xi)>2n-2. 
\end{equation}
Then for all large enough $k$, the polynomial $P_k$ 
from Definition \ref{newdef} is irreducible 
of degree exactly $n$. Consequenctly, for any large $k$, 
the set $\mathscr{V}_{k-1} \cup \mathscr{V}_k$ is linearly independent and thus spans a hyperplane
in the space of polynomials of degree at most $2n-2$.
Moreover we have
\begin{equation}  \label{eq:horse}
    \frac{ w_n(\xi) }{ \widehat{w}_n(\xi) } \le
    \frac{n-1}{ \widehat{w}_n(\xi)-n }.
\end{equation}
\end{lemma}

Some claims of the lemma
originate in results from \cite{buschlei}.
We aim to prove Theorem \ref{H} by contradiction, assuming throughout we had 
\begin{equation}  \label{eq:assuan}
\widehat{w}_n(\xi)> \beta_n
\end{equation}
for some $\xi$ 
and showing that it is impossible. Therefore, as $\beta_n>2n-2$, the condition \eqref{eq:2n} is not restrictive. The same applies to Theorem~\ref{T2}.
 It will occur
frequently. 

The next, new lemma extends the claim about the dimension of
unions of consecutive $\mathscr{V}_k$ from Lemma~\ref{ichalt}. Thereby it
avoids case 2 from the proof of \cite[Theorem 1.1]{2018}
and gives us more flexibility. The proof uses similar arguments.
It can be interpreted as a partial result towards proving the conditional
bounds from~\cite[\S~2]{2018}, however per se it does not lead to an improvement in this framework. 

\begin{lemma}  \label{lemur}
Let $n\ge 2$ be an integer and $\xi$ a real number. 
Assume \eqref{eq:2n} holds.
    For all large enough, good $k$ in sense of Definition~\ref{good}, 
   the polynomials from the union 
    \[
    \mathscr{R}_k:= \mathscr{V}_{k-1} \cup \mathscr{V}_k \cup \mathscr{V}_{k+1}
    \]
    span the space of polynomials of degree at most $2n-2$.
\end{lemma}

\begin{remark}
The linear independence assumption of the lemma is clearly also needed, 
else the union spans a space of dimension at most $2n-2$ (with equality
if \eqref{eq:2n} holds).
\end{remark}

\begin{remark}  \label{remb}
 The proof below shows that the claim holds for any three linearly independent irreducible polynomials of exact degree $n$, in particular upon \eqref{eq:2n} by Lemma \ref{ichalt} for any linearly independent triple of best approximation polynomials.
\end{remark}

\begin{proof} By Lemma \ref{ichalt}, it suffices to show that some polynomial from $\mathscr{V}_{k+1}$
does not belong to the span of $\mathscr{V}_{k-1}\cup \mathscr{V}_{k}$. If $P_{k+1}\in \mathscr{V}_{k+1}$ is not contained in this space, then
we are done. So assume it does belong to the span. Then we can write
\begin{equation} \label{eq:yYy}
    P_{k+1}= Q_1 P_{k-1} + Q_2 P_{k}
\end{equation}
for rational polynomials $Q_i=Q_{i,k}\in \mathbb{Q}[T]$
of degrees at most
\[
m:=\max_{i=1,2} \deg Q_i\le n-2.
\]
On the other hand, $m \ge 1$ follows from the linear independence 
assumption of the lemma. 
Without loss of generality, assume $\deg Q_1=m$, the other case works analogously. Then consider
$T^{n-m-1} P_{k+1}$ which has degree at most $2n-2$ so it lies in $\mathscr{V}_{k+1}$. 
Again if it does not lie in the span 
of $\mathscr{V}_{k-1} \cup \mathscr{V}_{k}$ we are done. So we can assume it does, which again means that
there is some identity
\begin{equation} \label{eq:V1}
T^{n-m-1} P_{k+1} = B_1 P_{k-1} + B_2 P_{k} 
\end{equation}
with rational polynomials $B_i\in \mathbb{Q}[T]$ of degrees 
\begin{equation}  \label{eq:Dgree}
    \deg B_i \le n-2, \qquad i=1,2.
\end{equation}
%The latter uses again that $P_{k-1}, P_k$ have 
%degree exactly $n$.
On the other hand, by \eqref{eq:yYy} we can write
\begin{equation} \label{eq:V2}
A_1 P_{k-1} + A_2 P_{k} = T^{n-m-1} P_{k+1},
\end{equation}
where $A_i\in \mathbb{Q}[T]$, $i=1,2$, 
are rational polynomials given as $A_i= T^{n-m-1}Q_i$. Note that $\deg A_1=(n-m-1)+m=n-1$. Hence
by \eqref{eq:Dgree}
also
\begin{equation}  \label{eq:grad}
    \deg(A_1-B_1)=n-1.
\end{equation}
Now from \eqref{eq:V1}, \eqref{eq:V2} we have an identity
\[
B_1 P_{k-1} + B_2 P_{k}  =  T^{n-m-1} P_{k+1}= A_1 P_{k-1} + A_2 P_{k}.
\]
This yields to an identity over $\mathbb{Q}[T]$ given as
\[
(A_1-B_1)P_{k-1} = -(A_2-B_2)P_{k},
\]
and multiplying with the common denominator we get an equality 
\[
(\tilde{A}_1-\tilde{B}_1)P_{k-1}= -(\tilde{A}_2-\tilde{B}_2)P_{k}
\]
over
integer polynomials $\tilde{A}_i, \tilde{B}_i\in \mathbb{Z}[T]$.
So since $P_{k-1}, P_{k}$ are distinct and irreducible of degree exactly $n$
by Lemma \ref{ichalt}, we must have that $P_{k}$ divides $\tilde{A}_1-\tilde{B}_1$ over $\mathbb{Z}[T]$, but since the latter has degree $n-1$ by \eqref{eq:grad}, this is impossible. The lemma is proved. 
\end{proof}

\subsection{On consecutive best approximations in two-dimensional subspaces} \label{on}

As alluded by \eqref{eq:conne}, for studying uniform exponents,
it is important to understand how fast consecutive best
approximations occur. Hence
we define
\begin{equation}  \label{eq:tk}
\tau_k= \frac{ \log H_{k} }{\log H_{k-1} } > 1, \qquad k\ge 2,
\end{equation}
that is
    \[
    H_{k} = H_{k-1}^{\tau_k}.
    \]
    All error terms below will be understood as $k\to\infty$.
Then, upon assuming \eqref{eq:2n}, by \eqref{eq:horse} and a well-known argument (see for example
\cite[Lemma 1]{ichmj}) we have
\begin{equation}  \label{eq:tauk}
    1< \tau_k \le \frac{ w_n(\xi) }{ \widehat{w}_n(\xi) } + o(1) \le
    \frac{n-1}{ \widehat{w}_n(\xi)-n }+ o(1).
\end{equation}
%We point out that all these estimates require no restriction for $k$.
Note that only the most right estimate of \eqref{eq:tauk} requires a condition
on $\xi$.
Furthermore as a consequence of \eqref{eq:conne} we have
\begin{equation}  \label{eq:T}
  H_{k}^{-w_n(\xi)/\tau_k - o(1) }= H_{k-1}^{-w_n(\xi)-o(1)} \ll  \vert P_{k-1}(\xi)\vert \ll H_{k}^{-\widehat{w}_n(\xi) + o(1) }.
\end{equation}
The next lemma, only required for the proof of Theorem~\ref{H},
estimates for how long a set of consecutive best approximation polynomials can lie in a two-dimensional subspace.
The proof follows closely ideas of Davenport and Schmidt \cite{davsh67}
and refines it in a quantitative way. As in \cite{davsh67},
it is not specific
to points on the Veronese curve and can be formulated for linear forms
in any $n$ real variables that are $\mathbb{Q}$-linearly independent 
together with $\{1\}$. 

\begin{definition} \label{elle}
	Given $n,\xi$, for $k\ge 2$ an integer, let $\ell=\ell(k)\ge k+1$ be the maximal integer so that $P_{k-1}, P_{k}, \ldots, P_{\ell-1}$ 
	from Definition~\ref{newdef}
	lie in a two-dimensional space (spanned by $P_{k-1}, P_{k}$).
\end{definition}

Note that $k$ is good in sense of Definition~\ref{good} if and only if $\ell(k)=k+1$.

\begin{lemma}  \label{le1}
    Let $n\ge 2$ be an integer, $\xi$ be a transcendental real number, 
    and let $(P_j)_{j\ge 1}$ be the sequence of best approximation polynomials for degree $n$ as in Definition \ref{newdef}. For $k\ge 2$ an integer,
    assume
      \begin{equation}  \label{eq:zweii}
     H_{k} > 2 \cdot H_{k-1}.
    \end{equation}
    Let $\ell=\ell(k)\ge k+1$ be as in Definition~\ref{elle} and
    let $v_{k-1}$ be as in \eqref{eq:nocce}.
    Then we have
    \begin{equation}  \label{eq:7}
        \frac{ \log H_{\ell-1} }{ \log H_{k} } \le \frac{ \frac{v_{k-1}}{\tau_k} -1 }{ \widehat{w}_n(\xi)-1 }+ o(1), \qquad k\to\infty.
    \end{equation}
    In particular
    \begin{equation} \label{eq:ratio}
    \frac{ \log H_{\ell-1} }{ \log H_{k} } \le \frac{ \frac{w_n(\xi)}{\tau_k} -1 }{ \widehat{w}_n(\xi)-1 }+ o(1), \qquad k\to\infty.
    \end{equation}
    Hence $P_{k-1}, P_{k}, P_{\ell}$ are linearly independent and
    \begin{equation} \label{eq:schwimm}
     \frac{ \log H_{\ell} }{ \log H_{k} } \le \frac{ \frac{w_n(\xi)}{\tau_k} -1 }{ \widehat{w}_n(\xi)-1 } \cdot \tau_{\ell}+ o(1), \qquad k\to\infty.
    \end{equation}
\end{lemma}

\begin{remark}
    The factor two in \eqref{eq:zweii} can be replaced by any value larger than $1$. The claim probably remains true without condition \eqref{eq:zweii} at all, however for us this condition will not cause major problems below.
\end{remark}

\begin{remark} \label{Rr}
If \eqref{eq:2n} holds, then
    we may use \eqref{eq:horse} to eliminate $w_n(\xi)$ and express the right hand sides in terms of $n, \widehat{w}_n(\xi), \tau_k, \tau_{\ell}$ only. For example \eqref{eq:schwimm} becomes 
    \[
    \frac{ \log H_{\ell} }{ \log H_{k} } \le \frac{ (n-1)\widehat{w}_n(\xi)-\tau_k(\widehat{w}_n(\xi)-n) }{\tau_k (\widehat{w}_n(\xi)-1)(\widehat{w}_n(\xi)-n) }\cdot \tau_{\ell} + o(1), \qquad k\to\infty.
    \]
    This will be used in the proof of Theorem \ref{B} below.
\end{remark}

We will only explicitly need \eqref{eq:schwimm} below.
It implies that we cannot have that
all but finitely many best approximations lie in a two-dimensional space, a claim from \cite[\S 4]{davsh67}. The latter is 
false for three-dimensional subspaces
and general real vectors of any dimension $n\ge 3$ 
(possibly none of the exceptions 
lie on the Veronese curve though), 
see \cite{nmg}.
Note further that the ratio in \eqref{eq:ratio} decreases to $1$ as $\tau_k$ approaches its upper bound 
from \eqref{eq:tauk}.

\begin{proof}
    Following the method of Davenport and Schmidt \cite[\S 4]{davsh67}
    stated for $n=2$ only but which generalizes to any $n$ by the same argument, we see the following: Whenever
    consecutive best approximations $P_{k-1}, P_{k}, \ldots, P_{\ell-1}$ lie in a two-dimensional space, denoting $x_j>0$ the leading coefficient of $P_j$ 
    %(non-zero by our observation $P_k$ being irreducible of %degree exactly $n$) 
    we have 
    \begin{equation}  \label{eq:1}
    \vert x_{k-1} P_{k}(\xi) - x_{k} P_{k-1}(\xi)\vert= \vert x_{\ell-2} P_{\ell-1}(\xi) - x_{\ell-1} P_{\ell-2}(\xi)\vert.
    \end{equation}
    The choice of the leading coefficient is not critical, the according identity remains true when choosing the coefficient of any other power (the same power for all $P_j$) by the same underlying determinant argument, not explicitly carried out in \cite{davsh67}.
    Hence we may assume 
    \[
    x_{k}= H_{k}
    \]
    as otherwise we choose instead the coefficient for the power that induces $H_{k}$ and the argument below works analogously.
    %by a minor modification of the argument of [cite Champagne, Roy], as %otherwise some birational transformation of $\xi$ that does not %affect exponents of approximation will lead to this scenario, upon %transitioning to a subsequence of indices $k$ if necessary. Indeed, %the small adjustment needed is to consider the polynomials $P_{k+1}$ %instead of a sequence with
    %$\log \vert P_j(\xi)\vert/\log H_j \to w_n(\xi)$.\\
    %
    %First assume for some fixed $\delta>1$, for large enough $k$ the heights %of two consecutive best approximations $P_k, P_{k+1}$ satisfy
    %\begin{equation}  \label{eq:zweii}
    % H_{k+1} > \delta \cdot H_k.
    %\end{equation}
    Now by \eqref{eq:zweii}, \eqref{eq:jopi} 
    and as $x_{k-1}\le H_{k-1}$ is obvious, 
    %writing
    %\[
    %\delta^{\prime}= 1- \frac{1}{\delta}>0,
    %\]
    this leads to 
    \begin{align}
    \vert x_{k-1} P_{k}(\xi) - x_{k} P_{k-1}(\xi)\vert &\ge 
    \vert x_{k} P_{k-1}(\xi)\vert - \vert x_{k-1} P_{k}(\xi)\vert  \nonumber
     \\ &\ge \vert H_{k} P_{k-1}(\xi)\vert - \vert H_{k-1} P_{k}(\xi)\vert  \nonumber    \\
     &\ge \vert H_{k} P_{k-1}(\xi)\vert - \vert (H_{k+1}/2) P_{k}(\xi)\vert
     \nonumber \\ &\ge \vert H_{k} P_{k-1}(\xi)\vert - \vert (H_{k}/2) P_{k-1}(\xi)\vert  \nonumber
    \\ &= \frac{1}{2}\cdot H_{k} \vert P_{k-1}(\xi)\vert \nonumber \\
    &= \frac{1}{2}H_{k}\cdot H_{k}^{- v_{k-1}/\tau_k }  \nonumber \\
    & = \frac{1}{2}H_{k}^{1- v_{k-1}/\tau_k }.   \label{eq:2}
    \end{align}
    On the other hand \eqref{eq:T}, \eqref{eq:jopi} imply
    \begin{align}
         \vert x_{\ell-2} P_{\ell-1}(\xi) - x_{\ell-1} P_{\ell-2}(\xi)\vert 
         &\le
      \vert x_{\ell-2} P_{\ell-1}(\xi)\vert + \vert x_{\ell-1} P_{\ell-2}(\xi)\vert \nonumber \\ &\le 2 x_{\ell-1} \vert P_{\ell-2}(\xi)\vert 
     \nonumber \\ & \le 2 H_{\ell-1} \vert P_{\ell-2}(\xi)\vert  \nonumber 
      \\ &\ll H_{\ell-1}^{1- \widehat{w}_n(\xi)+o(1) }. \label{eq:3}
    \end{align}
    Combining the three claims \eqref{eq:1}, \eqref{eq:2}, \eqref{eq:3} gives the estimate \eqref{eq:7}. Since $v_{k-1}\leq w_n(\xi)+o(1)$ as $k\to\infty$ by \eqref{eq:conne}, we infer \eqref{eq:ratio}. The last claim \eqref{eq:schwimm} follows in turn from \eqref{eq:ratio}
    %with \eqref{eq:tauk}
    % applied to index $\ell-1$ 
    via
    \[
    \frac{ \log H_{\ell} }{ \log H_{k} }=  
    \frac{ \log H_{\ell-1} }{ \log H_{k} } \cdot \frac{ \log H_{\ell} }{ \log H_{\ell-1} }
    \le \frac{ \frac{w_n(\xi)}{\tau_k} -1 }{ \widehat{w}_n(\xi)-1 } \cdot \tau_{\ell}+ o(1), \qquad k\to\infty.
    \]
    The lemma is proved.
%
 %   Now finally we justify why we may assume \eqref{eq:zweii}. 
  %  A direct consequence of a more general result by Bugeaud and Laurent [cite]
   % is that $H_{k+2^{n+1}} \ge u\cdot H_k$ for some absolute constant %$u=u(n)>1$
   % and all $k$. Let $\tau= (1+u^{1/2^{n+1}})/2>1$ the arithmetic
   % mean of $1$ and $u^{1/2^{n+1}}>1$. Then for some $0\le j\le 2^{n+1}-1$, we %must have $H_{k+j+1} > \tau H_{k+j}$. Let $j$ be minimal
   % with this property. We may assume $j\ge 1$, the case $j=0$ was dealt with
   % above.
   % Since the reverse inequality 
   % $H_{k+i+1}\le \tau H_{k+i}$
   % holds for smaller indeices $0\le i\le j-1\le 2^{n+1}-2$ and $n, \tau$ are %fixed, it is clear that $H_{j+k}\ll_n H_k$, in particular
    %\begin{equation} \label{eq:poten}
    %    \log H_{k+j} \le \log H_{k} + o(1), \qquad k\to\infty.
    %\end{equation}
    %But on the other hand, by definition of $j$
    %now we have the
    %condition \eqref{eq:zweii} satisfied for $H_{k+j}$ instead of $H_k$, 
    %with $\delta=\tau$. Thus the calculations above for
    %the conditional case apply and yield analogously
    %\[
    %\frac{ \log H_{\ell-1} }{ \log H_{k+j+1} } \le \frac{ \frac{v_{k+j}}%{\tau_{k+j} } -1 }{ \widehat{w}_n(\xi)-1 }+ o(1).
    %\]
    %However, by \eqref{eq:poten} clearly \eqref{eq:7} follows.
\end{proof}

\section{Parametric geometry of numbers: Introduction and a lemma} \label{pgn}
Our final prerequisite lemma will be formulated in the language of parametric
geometry of numbers, which will also
be used in the proofs of the main claims Theorem \ref{A}, \ref{B} below. 
We only give a brief summary of the
most important notation.
We consider the classical lattice point problem
induced by linear form (polynomial) approximation 
to $(\xi,\xi^2, \ldots,\xi^{2n-2})$, as in \cite{2018}.
Concretely, we consider the successive minima functions 
$\kappa_j(Q), 1\le j\le 2n-1$, with respect
to the lattice and parametrised family of convex bodies
of constant volume
\begin{align*}
\Lambda_{\xi}&=\{  (a_1, \ldots, a_{2n-2}, a_0+ a_1\xi +\cdots+a_{2n-2}\xi^{2n-2})\in \mathbb{R}^{2n-1}: a_i\in \mathbb{Z} \}, \\ K(Q)&= [-Q^{\frac{1}{2n-2}}, Q^{\frac{1}{2n-2}}]^{2n-2}\times [-Q^{-1},Q^{-1}]\subseteq \mathbb{R}^{2n-1}, \qquad Q>1,
\end{align*}
in $\mathbb{R}^{2n-1}$ and we derive the parametric functions
\[
L_{j}(q)= \log \kappa_{j}(e^q), \qquad q>0,\; 1\le j\le 2n-1.  
\]
These are piecewise linear with slopes among $\{ -1/(2n-2), 1\}$,
as locally they are realised by the trajectory
of some integer polynomial $P$ of degree
at most $2n-2$ and height $H_P$, defined as
\begin{equation} \label{eq:Trajan}
L_P(q)= \max\left\{   \log H_P - \frac{q}{2n-2} ,\; \log \vert P(\xi)\vert + q\right\}.
\end{equation}
%where $P$ is meant to depend on $q$.
In other words, for each $q\ge 0$ there are linearly independent polynomials $P^{(1)}, \ldots, P^{(2n-1)}$ as above (depending on $q$) with integer coefficients and of heights $H^{(1)}, \ldots,H^{(2n-1)}$, so that
\begin{equation}  \label{eq:E}
L_j(q)=L_{P^{(j)}}(q)= \max\left\{   \log H^{(j)} - \frac{q}{2n-2} ,\; \log \vert P^{(j)}(\xi)\vert + q\right\}, \quad 1\le j\le 2n-1.
\end{equation}
The polynomial $P^{(1)}$ minimizes $L_P$ over all 
relevant choices of $P$, hence 
\begin{equation*} 
L_1(q)= \min L_P(q)= \min \max\left\{  \log H_P - \frac{q}{2n-2} ,\; \log \vert P(\xi)\vert + q \right\}
\end{equation*}
with minimum taken over all non-zero integer polynomials $P$ of degree
at most $2n-2$. 

In Lemma \ref{le2} as well as in Theorems \ref{A}, \ref{B} 
below, we will
consider the trajectories \eqref{eq:Trajan} for $P=P_j$
the best approximation polynomials
of degree at most $n$ (which their naturally inclusion 
in the set of polynomials of degree at most $2n-2$) as in 
Definition \ref{newdef}, hence
\begin{equation} \label{eq:121}
    L_{P_j}(q)= \max\left\{   \log H_j - \frac{q}{2n-2} ,\; \log \vert P_{j}(\xi)\vert + q\right\}, \qquad j\ge 1.
\end{equation}
More generally, we will consider $L_P(q)$ for $P\in \mathscr{V}_k$ as defined in \eqref{eq:Vk}. 
We recall that Minkowski's Second Convex Body Theorem is equivalent to
\begin{equation}  \label{eq:MCBT}
    \left\vert\sum_{j=1}^{2n-1} L_{j}(q)\right\vert= O(1).
\end{equation}
The constant depends on $n$ only.
We recall that the quality of approximation $\vert P(\xi)\vert$ 
induced by some integer polynomial $P$ is essentially encoded
by the quotient $L_P(q)/q$ at its
unique global minimum point $q$ where
the rising and decaying part of the trajectory $L_P$ in \eqref{eq:Trajan} coincide. Now we can finally state our lemma.

\begin{lemma}  \label{le2}
Let $n\ge 2$ be an integer and $\xi$ be a real number satisfying
\eqref{eq:2n}.
    Let $P_{k-1}, P_{k}$ be two consecutive minimal polynomials
    with respect to approximation
    to $(\xi,\xi^2,\ldots,\xi^n)$
    as in Definition \ref{newdef}.
     Consider 
the combined graph
with respect to $(\xi,\xi^2,\ldots,\xi^{2n-2})$, with induced piecewise linear functions $L_{P_j}$ as in \eqref{eq:121} and successive minima 
functions $L_1,\ldots,L_{2n-1}$.
Let $q_{k}$ be the unique point where $L_{P_{k-1}}(q_{k})= L_{P_{k}}(q_{k})$
and $s_{k}< q_{k}$ be the place where $L_{P_{k-1}}$ is minimized, i.e. so that $L_{P_{k-1}}(s_{k})= \min_{q>0} L_{P_{k-1}}(q)$.
Then 
    \[
L_{2n-1}(q_{k}) \ge -(2n-2)\cdot L_{P_{k-1}}(q_{k}) + \frac{2n^2-5n+2}{2n-2}(q_{k}-s_{k}) - O(1).
\]
\end{lemma}

The gain compared to the method in \cite[Theorem 1.1]{2018} 
is the non-negative expression
$((2n^2-5n+2)/(2n-2))\cdot (q_{k}-s_{k})$. If $\tau_k$ exceeds some $\theta>1$ strictly, then the estimate $q_{k}-s_k\gg_{\theta} q_{k}$ can be verified, leading to an improvement for $n\ge 3$. On the other hand,
the identity $2n^2-5n+2=0$ for $n=2$ naturally
agrees with the bound \eqref{eq:DSC} for $n=2$ from  \cite{davsh} (also obtained with different proofs in \cite{buschlei} and \cite{2018}) 
being optimal.

The proof of Lemma \ref{le2} is rather straightforward.

\begin{proof}
Since the slope of any trajectory is at least $-1/(2n-2)$, we clearly have
\begin{equation}  \label{eq:lesbos}
    L_{2n-1}(q_{k}) \ge L_{2n-1}(s_k) - \frac{q_{k} - s_k}{2n-2}.
\end{equation}
On the other hand, it is clear that in the
interval $(s_k, q_{k})$ the function $L_{P_{k-1}}$ increases with slope $1$, whereas $L_{P_{k}}$ has slope $-1/(2n-2)$. Thus
\begin{equation} \label{eq:RF}
L_{P_{k-1}}(q_{k}) +  L_{P_{k}}(q_{k}) \ge L_{P_{k-1}}(s_k) +  L_{P_{k}}(s_k) + (1-\frac{1}{2n-2})(q_{k}-s_k).
\end{equation}
Since we assume \eqref{eq:2n},
we may apply Lemma \ref{ichalt} to see
that $\mathscr{V}_{k-1}\cup \mathscr{V}_{k}$ span a space of dimension $2n-2$ (hyperplane).
We may assume $\xi\in (0,1)$ so that $\vert P_j(\xi)\vert= \max \vert P(\xi)\vert$ with maximum taken over $P\in \mathscr{V}_j$, for all $j$. Since the heights of all $P\in \mathscr{V}_j$ coincide (all equal to $H_j$), via \eqref{eq:Trajan}
in particular this implies
$L_{P_j}(q)$ maximise $L_P(q)$ on $(0,\infty)$
among all $P\in \mathscr{V}_j$, for all $j \ge 1$, i.e.
\begin{equation}  \label{eq:maxi}
    L_{P_j}(q)= \max_{P\in \mathscr{V}_j  } L_P(q), \qquad q>0,\; j\ge 1.
\end{equation} 
Then \eqref{eq:RF} implies
\begin{align*}
\sum_{j=1}^{2n-2} L_{j}(s_k) &\le 
\sum_{P\in \mathscr{V}_{k-1}\cup \mathscr{V}_{k} } L_{P}(s_k)
\\ &\le(n-1) (L_{P_{k-1}}(s_k) + L_{P_{k}}(s_k))
\\ &\le (n-1)\left(L_{P_{k-1}}(q_{k}) +  L_{P_{k}}(q_{k}) - (1-\frac{1}{2n-2})(q_{k}-s_k)  \right) \\
&= (n-1)\left(2 L_{P_k}(q_{k}) - (1-\frac{1}{2n-2})(q_{k}-s_k)  \right)
\\&= (2n-2) L_{P_k}(q_{k}) - (n-\frac{3}{2})(q_{k}-s_k).
\end{align*}
Thus by \eqref{eq:MCBT}
\begin{align}  \label{eq:jaje}
    L_{2n-1}(s_k)&\ge - \sum_{j=1}^{2n-2} L_{j}(s_k) - O(1) \nonumber \\ &\ge
    -(2n-2) L_{P_k}(q_{k}) + (n-\frac{3}{2})(q_{k}-s_k) - O(1).
\end{align}
Now combining \eqref{eq:lesbos}, \eqref{eq:jaje} gives 
\begin{align*}
L_{2n-1}(q_{k})&\ge -(2n-2) L_{P_{k-1}}(q_{k}) + (n-\frac{3}{2}-\frac{1}{2n-2})(q_{k}-s_k) - O(1)\\ &= -(2n-2) L_{P_{k-1}}(q_{k}) +\frac{2n^2-5n+2}{2n-2}(q_{k}-s_k) - O(1),
\end{align*}
the claimed inequality.
\end{proof}

\section{A bound for $\widehat{w}_n(\xi)$ increasing in $\tau_k$}

The method of the proof of \cite[Theorem 1.1]{2018} 
essentially implicitly 
gives the following estimate.

\begin{theorem}    \label{A}
    Assume $n\ge 2$ is an integer and $\xi$ is a real number. Let
    $k\ge 2$ be an integer and $\tau_k$ be defined as in \eqref{eq:tk}
    and $\mu_{k}$ as in \eqref{eq:nocce}. Then
   for all good $k$ in terms of Definition~\ref{good}, 
    we have
    \[
    \tau_{k+1} \ge \mu_k - (2n-3).
    \]
    By \eqref{eq:conne} we conclude that for any $\varepsilon>0$
    and all good $k\ge k_0(\varepsilon)$ we have 
    \begin{equation} \label{eq:Turmf}
    \tau_{k+1} \ge \widehat{w}_n(\xi) - (2n-3) - \varepsilon.
    \end{equation}
\end{theorem}

\begin{proof}
We may assume \eqref{eq:2n} as otherwise the claim is obvious by $\tau_k>1$. We use the parametric geometry of numbers setup introduced
in \S \ref{pgn}.
    Very similarly to 
    the proof of \cite[Theorem 1.1]{2018}, at the position $q_{k}$
    where the trajectories $L_{P_{k-1}}$ and $L_{P_{k}}$ meet we have
    \begin{equation}  \label{eq:ducke}
    L_{2n-2}(q_k)\le L_{P_k}(q_{k}) = \frac{2n-2-\mu_k }{(2n-2)(1+\mu_k)}\cdot q_k.
    \end{equation}
    The inequality stems from Lemma \ref{ichalt}.
    For the latter identity notice that it is clear that $L_{P_k}$ decreases
    with slope $-1/(2n-2)$ up to this point $q_k$ whereas $L_{P_{k-1}}$
    increases with slope $1$ on some left neighborhood of $q_k$.
    Hence by \eqref{eq:121} we have
      \begin{equation} \label{eq:mostr}
    \log |P_{k-1}(\xi)| + q_k = L_{P_{k-1}}(q_k) = L_{P_k}(q_k)= \log H_k - \frac{q_k}{2n-2}
    \end{equation}
    and the identity can be derived with some calculation
    and the definition of $\mu_k$. We remark that by \eqref{eq:conne}, we may deduce
    \begin{equation*} 
    L_{2n-2}(q_k)\le L_{P_k}(q_{k}) \le \left(\frac{2n-2-\widehat{w}_n(\xi) }{(2n-2)(1+\widehat{w}_n(\xi))} + o(1)\right)\cdot q_k,\qquad k\to\infty,
    \end{equation*}
    which is precisely \cite[(49)]{2018}. 
    Combining \eqref{eq:ducke} 
    with the most right identity of \eqref{eq:mostr} we get
     \begin{equation} \label{eq:RE}
         \log H_k \le \left(\frac{2n-2-\mu_k }{(2n-2)(1+\mu_k)} + \frac{1}{2n-2}\right)\cdot q_k.
     \end{equation}
    On the other hand, again very similar to \cite{2018} from \eqref{eq:MCBT}, \eqref{eq:ducke} we get
    \[
    L_{2n-1}(q_k)\ge -(2n-2) L_{2n-2}(q_k) - O(1)\ge \frac{ \mu_k - 2n +2 }{1+\mu_k } \cdot q_k.
    \]
    Let
    $P$ of height $H_P$ denote the integer polynomial realizing the last minimum at $q_k$, i.e. 
    $L_P(q_k)=L_{2n-1}(q_k)$.
    Since $L_{P}$ is easily seen to decrease until $q_k$ with slope $-1/(2n-2)$ (see the proof of \cite[Theorem 1.1]{2018} for details) by \eqref{eq:Trajan} we get
    \begin{align}  
    \log H_P &=  L_{P}(q_k) + \frac{ q_k }{2n-2}  \nonumber \\
    &=L_{2n-1}(q_k) + \frac{ q_k }{2n-2} \ge \left(\frac{ \mu_k - 2n +2 }{1+\mu_k } + \frac{1}{2n-2} \right) \cdot q_k. \label{eq:AA1}
    \end{align}
    By the linear independence assumption and
    Lemma \ref{lemur}, we must have
    \begin{equation}  \label{eq:AA2}
    H_{k+1}\ge H_P.
    \end{equation}
    Assume otherwise $H_{k+1}<H_P$. By Lemma~\ref{lemur}
    we have $T^jP_{k+1}(T)$ not in the span of $\mathscr{V}_{k-1}\cup \mathscr{V}_k$ for some $0\le j\le n-2$.
    We may assume $j=0$ the other cases work analogously.
    Then since both $L_{P_{k+1}}$ and $L_P$ clearly decrease until
    $q_k$, from \eqref{eq:Trajan} we get 
    \[
     L_{P_{k+1}}(q_k) = \log H_{k+1}-\frac{q_k}{2n-2} < \log H_P - \frac{q_k}{2n-2} = L_{P}(q_k) = L_{2n-1}(q_k),
    \]
    contradicting the linear independence of $\mathscr{V}_{k-1}\cup \mathscr{V}_k\cup \{ P_{k+1}\}$. Combining the estimates \eqref{eq:RE}, \eqref{eq:AA1}, \eqref{eq:AA2}, we get
    \[
    \tau_{k+1}= \frac{ \log H_{k+1} }{\log H_k } \ge \frac{ \log H_P }{\log H_k} \ge
    \frac{\frac{ \mu_k - 2n +2 }{1+\mu_k } + \frac{1}{2n-2} }{  \frac{2n-2-\mu_k }{(2n-2)(1+\mu_k)} + \frac{1}{2n-2} } .
    \]
    Finally, the right hand side can be simplified to the desired expression $\mu_k-(2n-3)$.
\end{proof}

\begin{remark}
    If we combine \eqref{eq:Turmf} with \eqref{eq:tauk} for index $k+1$, we get the upper bound for $\widehat{w}_n(\xi)$
    from \cite[Theorem 1.1]{2018}. Recall this bound is 
    weaker than $\alpha_k$ from~\S~\ref{intro}.
\end{remark}

Note that we only used Lemma \ref{ichalt} and Lemma \ref{lemur},
the latter can even be omitted
upon introducing a more complicated argument, see Case 2 from the proof of 
\cite[Theorem 1.1]{2018}.

%\frac{\frac{ \widehat{w}_n(\xi) - 2n +2 }{1+\widehat{w}_n(\xi) } + \frac{1}{2n-%2} }{  \frac{2n-2-\widehat{w}_n(\xi) }{(2n-2)(1+\widehat{w}_n(\xi))} + 
%\frac{1}{2n-2} } - o(1)=

From $\tau_k$ given in \eqref{eq:tk},
derive
\begin{equation} \label{eq:kauen}
   \overline{\tau}:= \limsup_{k\to\infty} \tau_k.
\end{equation}
For the concern of Theorem~\ref{H},
we will only apply Theorem \ref{A} in form of the following corollary.

\begin{corollary} \label{IMPCOR}
    Assume $n\ge 2$ is an integer and $\xi$ is a real number satisfying \eqref{eq:2n}. Then
    \[
    \widehat{w}_n(\xi)\le \overline{\tau}+2n-3.
    \]
\end{corollary}

\begin{proof}
    Since infinitely many $k$ are good (as follows from Lemma \ref{le1}, or Davenport and Schmidt's method~\cite{davsh67} extended to general $n$), considering such a subsequence,
the estimate follows directly from Theorem \ref{A}.
\end{proof}

\section{A bound for $\widehat{w}_n(\xi)$ decreasing in $\tau_k, \tau_{\ell}$}

The main new contribution to improve the upper bound for $\widehat{w}_n(\xi)$ is the following
slightly technical result whose proof uses all three new lemmas from \S~\ref{prepa} and \S~\ref{pgn}.

\begin{theorem}  \label{B}
Let $n\ge 2$ be an integer and $\xi$ be a real number satisfying \eqref{eq:2n}.
For any $\varepsilon>0$, there exists $k_0(\varepsilon)>0$ such that
    for any integer $k\ge k_0(\varepsilon)$ satisfying \eqref{eq:zweii}, if $\ell=\ell(k)$ is as in Definition~\ref{elle} and with $\tau_{.}$
    defined in \eqref{eq:tk}, we have
    \[
    \Theta_n=\Theta_n(\widehat{w}_n(\xi), \tau_k, \tau_{\ell}) \le \varepsilon
    \]
    where
 %   \begin{align*}
 %   \Theta &= ((2n-2)\tau_k\cdot %\widehat{w}_n(\xi) - 
%(2n^2-3n+2)\tau_k - (2n^2-5n+2)) \\ %&\cdot(\widehat{w}_n(\xi)-1)(\widehat{w}_n(\xi)-n)
%\\ &- ((n-1)\widehat{w}_n(\xi)-
%\tau_k(\widehat{w}_n(\xi)-n))\cdot %(2n-1)\tau_{\ell}\\ &
%+\tau_k (\widehat{w}_n(\xi)+1)(\widehat{w}_n(\xi)-%1)(\widehat{w}_n(\xi)-n).
    %\end{align*}
    %that is
    \[
    \Theta_n= d_3 \widehat{w}_n(\xi)^3 + d_2 \widehat{w}_n(\xi)^2 + d_1 \widehat{w}_n(\xi) + d_0
    \]
    with $d_j= d_j(n,\tau_k, \tau_{\ell})$ given by
    \begin{align*}
           d_3&= \tau_k, \\
    d_2&=- (2n\tau_k +n-2),\\
    d_1&=\tau_k\tau_{\ell} + (n^2+n-1)\tau_k+ (1-n)\tau_{\ell} + n^2-n-2, \\  
    d_0&= -n\cdot ( (n-1+\tau_{\ell})\tau_k + n-2 ).
    \end{align*}
 
  %  eeghgrhhah
  %  \begin{align*}
%&\frac{ (n-1)\tau_k\widehat{w}_n(\xi) - (2n^2-5n+4)\tau_k - n +2  }{ (n-%1)\tau_k(1+\widehat{w}_n(\xi))  } \\ &\le  \frac{ (n-1)\widehat{w}_n(\xi)-%\tau_k(\widehat{w}_n(\xi)-n) }{\tau_k (\widehat{w}_n(\xi)-1)
%(\widehat{w}_n(\xi)-n) } \cdot \frac{(2n-1)\tau_{\ell} }{(2n-2)
%(1+\widehat{w}_n(\xi))  }  - \frac{1}{2n-2} + \epsilon.
%\end{align*}
\end{theorem}

\begin{remark}
    Again the assumption \eqref{eq:zweii} is probably not necessary.
\end{remark}

%First we notice that we may apply Lemma \ref{le1} in our setting.
%
%\begin{proposition}  \label{protokoll}
%    Assume $\widehat{w}_n(\xi)>2n-2$ and $P_{k-1}, P_k, P_{k+1}$ are linearly
%    independent and \eqref{eq:tri} holds. Then for fixed $\theta=\theta(n)>1$ 
 %   as $k\to\infty$ we have
 %   $$\tau_k \geq \frac{\widehat{w}_n(\xi)}{2n-2}-o(1)>\theta-o(1)>1.$$
 %   In particular \eqref{eq:zweii} holds for large enough $k$.
%\end{proposition}

%\begin{proof}
% Assume the reverse inequality.
%    Then for the parameter $X=H_{k+1}$ by Lemma \ref{lemur} we get $2n-1$ %linearly independent polynomials spanning the space of degree at most $2n-2$ %polynomials, with heights at most $X$ and evaluations at $\xi$ smaller than %$X^{-(2n-2)-\delta}$ for fixed $\delta>0$ and all large $k$. This easily %contradicts Minkowski's Second Convex Body Theorem. 
%\end{proof}
%
%\begin{remark}
%Moreover, for certain arbitrarily large $k$ we have
% \[
%    \tau_k\geq \frac{ w_n(\xi) - n + 1 }{n} - o(1) \ge \theta^{\prime} - o(1) %> 1.
%    \]
%This can be obtained from [cite ich mcnt 2022 Proposition ...] based on %Liouville's inequality and structural results from [cite marnat mosh]
%implying that
%\begin{equation}  \label{eq:tri}
%-\frac{\log P_k(\xi)}{H_k} = w_n(\xi) - o(1).
%\end{equation} 
%holds for infinitely many $k$ where $P_{k-1}, P_k, P_{k+1}$ 
%are linearly independent,
%further noticing that the bound by Marnat and Moshchevitin [cite] that %$\widehat{w}_n(\xi)>2n-2$ implies $w_n(\xi)>2n-1$ strictly. 
%\end{remark}

We prove the theorem.

\begin{proof}[Proof of Theorem \ref{B}]
We again consider the combined graph
with respect to approximation to
$(\xi,\xi^2,\ldots,\xi^{2n-2})$, with notation as in \S~\ref{pgn}.
Take a sequence of $k$ satisfying \eqref{eq:zweii}.
Recall $q_{k}$ is the first coordinate of the meeting point of $L_{P_{k-1}}$ and $L_{P_{k}}$.
Define $\Omega_k<0$ by
\begin{equation}  \label{eq:dass}
\Omega_k := \frac{ L_{P_{k-1}}(q_{k}) }{q_{k}} = \frac{ L_{P_{k}}(q_{k}) }{q_{k}}.
\end{equation}
As in the proof of Theorem \ref{A}, again by \cite[(49)]{2018} we have
\begin{equation} \label{eq:EQU}
\Omega_k \le \frac{2n-2-\widehat{w}_n(\xi) }{(2n-2)(1+\widehat{w}_n(\xi))} + o(1), \qquad k\to\infty.
\end{equation}
From Lemma \ref{le2} similarly to the proof of Theorem \ref{A} we get
\begin{align}
        L_{2n-1}(q_{k})&\ge -(2n-2) L_{2n-2}(q_{k}) + (q_{k}-s_k)\frac{2n^2-5n+2}{2n-2} - O(1) \nonumber \\ &\ge -(2n-2)\Omega_k\cdot q_{k} + (q_{k}-s_k)\frac{2n^2-5n+2}{2n-2}- O(1)\nonumber \\ &= 
        \left(\frac{2n^2-5n+2}{2n-2}- (2n-2)\Omega_k\right)q_{k}-\frac{2n^2-5n+2}{2n-2} s_k - O(1). \label{eq:hun}
\end{align}
We next provide an estimate for $s_k$ in terms of $q_{k}$ and $\tau_k$. 
Notice first that the definition of $q_{k}$ implies
that $L_{P_{k}}$ decays with slope $-1/(2n-2)$ on the interval $(0,q_{k})$, so \eqref{eq:Trajan} and \eqref{eq:dass} imply
\begin{align}
\log H_{k}&= L_{P_{k}}(q_{k}) + \frac{q_{k}}{2n-2} = q_{k}\cdot \left(\frac{1}{2n-2}+ \Omega_k \right).   \label{eq:cut} 
\end{align} 
%We may assume equality in the above inequality, otherwise the bounds %become only stronger. 
Since $L_{P_{k-1}}$ decreases with slope $-1/(2n-2)$ up to $s_k$ and
increases with slope $1$ from $s_k$ to $q_{k}$ by \eqref{eq:Trajan}, we get
\begin{align}
L_{P_{k-1}}(q_{k})&= \log H_{k-1} - \frac{ s_k }{2n-2} + (q_{k} - s_k) \nonumber \\ &=
\frac{ \log H_{k} }{\tau_k} - \frac{ s_k }{2n-2} + (q_{k} - s_k) \nonumber \\  &\le
 \left( \frac{\frac{1}{2n-2}+ \Omega_k}{\tau_k}  \right)q_{k} - \frac{ s_k }{2n-2} + (q_{k} - s_k) \nonumber \\ 
&=\left(\frac{\frac{1}{2n-2}+ \Omega_k}{\tau_k} + 1 \right) q_{k} - \frac{2n-1}{2n-2}s_k. \label{eq:1a}
\end{align}
On the other hand by \eqref{eq:dass} we have
\begin{equation} \label{eq:1b}
L_{P_{k}}(q_{k})=L_{P_{k-1}}(q_{k})= \Omega_k q_{k}.
\end{equation}
Comparing \eqref{eq:1a}, \eqref{eq:1b}
and solving for $s_k$ we get
\begin{align}
    s_k \le \frac{2n-2}{2n-1}\cdot \left( \frac{\frac{1}{2n-2}+ \Omega_k}{\tau_k} + 1 - \Omega_k \right) \cdot q_{k}. \label{eq:sk}
\end{align}
Note that the bound for $s_k$ becomes $q_{k}$ if $\tau_k=1$, an intuitively result.
Inserting \eqref{eq:sk} in \eqref{eq:hun} and dividing by $q_{k}$ we get 
\begin{align*}
\frac{L_{2n-1}(q_{k})}{q_{k}} &\ge \frac{2n^2-5n+2}{2n-2}-(2n-2)\Omega_k \\
&- \frac{2n^2-5n+2}{2n-2}\cdot \frac{2n-2}{2n-1}\cdot \left( \frac{\frac{1}{2n-2}+ \Omega_k}{\tau_k} + 1 - \Omega_k \right)\\
&=\frac{2n^2-5n+2}{2n-2}-(2n-2)\Omega_k \nonumber \\ &- \frac{2n^2-5n+2}{2n-1}\cdot \left( \frac{\frac{1}{2n-2}+ \Omega_k}{\tau_k} + 1 - \Omega_k \right).
\end{align*}
Now %using $\tau_k>1$ 
by differentiating the bound with respect to $\Omega_k$ and observing $(2n-2)(2n-1)=4n^2-6n+2\ge 2n^2-5n+2$ for $n\ge 1$, we readily check that this lower bound decreases in the variable $\Omega_k$. Hence we may assume 
asymptotic equality in \eqref{eq:EQU} as $k\to\infty$. Inserting, after rearrangements 
this leads to the estimate 
\begin{equation}  \label{eq:UP}
\frac{L_{2n-1}(q_{k})}{q_{k}} \ge \frac{ (2n-2)\tau_k\cdot \widehat{w}_n(\xi) - 
(2n^2-3n+2)\tau_k - (2n^2-5n+2)
 }{ (2n-2)\tau_k(1+\widehat{w}_n(\xi))  }- \frac{\epsilon}{2},
\end{equation}
for $\epsilon>0$ and any large enough $k\ge k_0(\epsilon)$.
 
\par
We now aim to find a reverse upper bound for $L_{2n-1}(q_{k})/q_{k}$. With $\ell=\ell(k)$ as in Definition~\ref{elle},
note that $P_{k-1}, P_{k}, P_{\ell}$ are linearly independent
and so by Lemma \ref{lemur} the polynomials $\mathscr{V}_{k-1} \cup \mathscr{V}_{k}\cup \mathscr{V}_{\ell}$ span the
space of polynomials of degree
at most $2n-2$ 
(Lemma \ref{lemur} holds for any linearly independent triple of $P_j$, see Remark \ref{remb}). Thus $T^j P_{\ell}(T)$ does not lie in the span 
of $\mathscr{V}_{k-1} \cup \mathscr{V}_{k}$ for some $0\le j\le n-2$. We may assume for simplicity $j=0$ so that
$P_{\ell}\in \mathscr{V}_{\ell}$ does
not belong to the span of $\mathscr{V}_{k-1} \cup \mathscr{V}_{k}$, the other cases work very similarly. So
$\mathscr{V}_{k-1} \cup \mathscr{V}_{k}\cup \{ P_{\ell}\}$ are linearly independent. Now
\[
L_{P_{\ell}}(q_{k})> 0> \max_{ P\in \mathscr{V}_{k-1} \cup \mathscr{V}_{k}}  L_{P}(q_{k}) = L_{P_k}(q_{k})
\]
is easily verified where the last identity comes from \eqref{eq:maxi}
and $L_{P_{k-1}}(q_{k})= L_{P_{k}}(q_{k})$ by definition of $q_{k}$.
Moreover, again by \eqref{eq:Trajan}, the function $L_{P_{\ell}}$ obviously decays with 
slope $-1/(2n-2)$ 
in the interval $(0,q_{k})$ as its minimum is taken at some $q_{\ell}>q_{k}$. Combining 
all these facts we get
\begin{align} \label{eq:Ff}
L_{2n-1}(q_{k}) \le \max_{ P\in \mathscr{V}_{k-1} \cup \mathscr{V}_{k} \cup \{P_{\ell}\} }  L_{P}(q_{k}) =
L_{P_{\ell}}(q_{k}) = \log H_{\ell} - \frac{ q_{k} }{2n-2}.  
\end{align}
To bound the right hand side we next estimate $H_{\ell}$. Note that we may apply Lemma \ref{le1}
by assumption \eqref{eq:zweii}. %Proposition \ref{protokoll}.
By applying in this order \eqref{eq:schwimm} of Lemma \ref{le1}, \eqref{eq:cut}, \eqref{eq:EQU}, 
and \eqref{eq:horse}, we get
\begin{align}
\log H_{\ell} &\le \left( \frac{ \frac{w_n(\xi)}{\tau_k} -1 }{ \widehat{w}_n(\xi)-1 } \cdot \tau_{\ell} +o(1)\right) \cdot \log H_{k}  \nonumber \\ &\le \left( \frac{ \frac{w_n(\xi)}{\tau_k} -1 }{ \widehat{w}_n(\xi)-1 } \cdot \tau_{\ell} \left( \frac{1}{2n-2} + \Omega_k \right) +o(1)\right) \cdot q_{k}
\nonumber \\ &\le
\left(\frac{ \frac{w_n(\xi)}{\tau_k} -1 }{ \widehat{w}_n(\xi)-1 } \cdot \tau_{\ell}\cdot \frac{2n-1}{(2n-2)(1+\widehat{w}_n(\xi))  }+o(1)\right) \cdot q_{k}
%\\&=  \frac{ w_n(\xi)/\tau_k -1 }{ \widehat{w}_n(\xi)-1 } \cdot \frac{2n-1}%{2(\widehat{w}_n(\xi)-n)(1+\widehat{w}_n(\xi))}\cdot q_{k+1} 
\nonumber \\ &\le \left(\frac{ (n-1)\widehat{w}_n(\xi)-\tau_k(\widehat{w}_n(\xi)-n) }{\tau_k (\widehat{w}_n(\xi)-1)(\widehat{w}_n(\xi)-n) } \cdot \frac{(2n-1)\tau_{\ell} }{(2n-2)(1+\widehat{w}_n(\xi))  }+o(1)\right)\cdot q_{k}. \label{eq:NOVUS}
%\\ &= \frac{ (2n-1)\cdot ((n-1-\tau_k)\widehat{w}_n(\xi)+n\tau_k) }{2\tau_k 
%(\widehat{w}_n(\xi)-1)(\widehat{w}_n(\xi)-n)^2(1+\widehat{w}_n(\xi))}\cdot %q_{k+1}.
\end{align}
See also Remark \ref{Rr}.
Inserting in \eqref{eq:Ff} and dividing by $q_{k}$ yields
\begin{equation}  \label{eq:DO}
\frac{L_{2n-1}(q_{k})}{q_{k}} \le \Delta- \frac{1}{2n-2} + \frac{\epsilon}{2} 
\end{equation}
where
\[
\Delta= \frac{ (n-1)\widehat{w}_n(\xi)-\tau_k(\widehat{w}_n(\xi)-n) }{\tau_k (\widehat{w}_n(\xi)-1)(\widehat{w}_n(\xi)-n) } \cdot \frac{(2n-1)\tau_{\ell} }{(2n-2)(1+\widehat{w}_n(\xi))  },
\]
for $\epsilon>0$ and $k\ge k_1(\epsilon)$ again.
\par

Comparing the lower bound \eqref{eq:UP} and the upper bound \eqref{eq:DO} obtained for $L_{2n-1}(q_{k})/q_{k}$ gives
a relation involving $n,\tau_k, \tau_{\ell}$ 
and $\widehat{w}_n(\xi)$ of the form 
\begin{align*}
&\frac{ (2n-2)\tau_k\cdot \widehat{w}_n(\xi) - 
(2n^2-3n+2)\tau_k - (2n^2-5n+2)
 }{ (2n-2)\tau_k(1+\widehat{w}_n(\xi))  } \\ &\le \frac{ (n-1)\widehat{w}_n(\xi)-\tau_k(\widehat{w}_n(\xi)-n) }{\tau_k (\widehat{w}_n(\xi)-1)(\widehat{w}_n(\xi)-n) } \cdot \frac{(2n-1)\tau_{\ell} }{(2n-2)(1+\widehat{w}_n(\xi))  }  - \frac{1}{2n-2}+\epsilon,
\end{align*}
for $k\ge \max\{ k_0(\epsilon),k_1(\epsilon)\}$.
Subtracting the right side from the left,
multiplying with the common denominator and some rearrangements lead to
the estimate in Theorem \ref{B} upon modifying $\epsilon$.
\end{proof}

\section{Deduction of Theorem \ref{H} from Theorems \ref{A}, \ref{B} } \label{pthm1}

Recall that we want to prove the claim indirectly, assuming \eqref{eq:assuan} holds for some $\xi$ 
and leading it to a contradiction. Hence, as $\beta_n>2n-2$, we can 
assume the condition \eqref{eq:2n} of Theorem \ref{B} is satisfied.
Recall $\overline{\tau}$ from \eqref{eq:kauen}.
Notice this quantity depends on $\xi, n$ only. We distinguish two cases.\\

Case 1: $\overline{\tau}=1$. Then Corollary \ref{IMPCOR} directly implies
$\widehat{w}_n(\xi)\le \overline{\tau}+2n-3 =2n-2< \beta_n$. \\

Case 2: $\overline{\tau}>1$. 
We may assume the sequence $\overline{\tau}<\infty$ 
in view of \eqref{eq:tauk}, otherwise again the strengthened 
bound $\widehat{w}_n(\xi)\le 2n-2 < \beta_n$ follows immediately.
Choose a sequence of values $k$ converging to $\overline{\tau}$. 
For any large enough $k$ in this sequence,
by assumption of Case 2 we have
$\tau_k>\theta>1$ is strictly bounded away from $1$, in particular
condition \eqref{eq:zweii} of Lemma \ref{le1} and Theorem \ref{B} holds. 
Thus for such $k$ we can apply Theorem \ref{B}. Note that its bound $\Theta_n$ increases as a function in the third argument $\tau_{\ell}$ (but decreases in the argument $\tau_k$). As $\tau_k\to \overline{\tau}$ 
and clearly $\tau_{\ell}\le \overline{\tau}+o(1)$ for $\ell= \ell(k)$ as in Definition~\ref{elle} as $k\to\infty$,
we may put $\tau_k=\tau_{\ell}=\overline{\tau}$ in the inequality claim of Theorem \ref{B} upon modifying $\epsilon$.
In the limit we may ignore this error term $\epsilon$ and we obtain
\begin{equation}   \label{eq:QQ}
\Theta_n(\widehat{w}_n(\xi), \overline{\tau}, \overline{\tau})\le 0.
\end{equation}
It can be checked that \eqref{eq:QQ} induces a bound
\[
\widehat{w}_n(\xi) \le F_n(\overline{\tau}) 
\]
for some decreasing function $F_n: (1,\infty)\to \mathbb{R}$. Combined with Corollary \ref{IMPCOR}
we get
\[
\widehat{w}_n(\xi) \le \min \{  \overline{\tau} + 2n-3 , F_n(\overline{\tau}) \}. 
\]
Now the left bound increases while the right bound decreases in $\overline{\tau}$, hence the maximum (worst case) of the above minimum
is obtained for the equilibrium $\overline{\tau}$ 
where the expressions coincide. 
Hence we can put 
\[
\overline{\tau}= \widehat{w}_n(\xi)-2n+3
\]
in the defining equation \eqref{eq:QQ} of $F_n$ and solve for equality
\[
\Theta_n(\widehat{w}_n(\xi), \overline{\tau}, \overline{\tau}) =
\Theta_n(\widehat{w}_n(\xi),\widehat{w}_n(\xi)-2n+3,\widehat{w}_n(\xi)-2n+3)= 0
\]
in the variable $\widehat{w}_n(\xi)$. After simplifications, 
this gives a quartic polynomial equation $Q_n(\widehat{w}_n(\xi))=0$ 
for $Q_n(T)$ precisely the polynomial in Theorem \ref{H}.
Its largest root $\beta_n$ is checked to be in the interval $(2n-2,2n-1)$ and so Theorem \ref{H} is proved.

\section{Sketch of the proof of Theorem~\ref{T2} }

We establish a variant of Theorem~\ref{B}. Assume first
hypothesis (i) of Theorem~\ref{T2} holds
and fix any large good $k$.
Then $\ell=k+1$ in Definition~\ref{elle}, hence
in \eqref{eq:NOVUS} within the proof of Theorem~\ref{B} 
the factor $(w_n(\xi)/\tau_k-1)/(\widehat{w}_n(\xi)-1)$
derived from Lemma~\ref{le1}
can be omitted (replaced by $1$). Proceeding
as in the proof of Theorem~\ref{B} this results in an estimate
\begin{equation} \label{eq:stein}
\tilde{\Theta}_n= \tilde{\Theta}_n(\widehat{w}_n(\xi), \tau_k, \tau_{\ell})  \le \epsilon,
\end{equation}
for $\epsilon>0$ provided that $k\ge k_0(\epsilon)$ is good, where 
\[
\tilde{\Theta}_n\!:=\! (2n-2)\tau_k \widehat{w}_n(\xi) - (2n^2-3n+2)\tau_k - (2n^2-5n+2) - \tau_k((2n-1)\tau_{\ell}-1-\widehat{w}_n(\xi)).
\]
This again defines an inequality $\widehat{w}_n(\xi)\le \tilde{F}_n(\tau_k,\tau_{\ell}) + o(1)$ with some function 
$\tilde{F}_n$ decreasing
in $\tau_k$ and increasing in $\tau_{\ell}$.
Now we estimate $\tau_{\ell}$ as in
\eqref{eq:tauk} in terms of $\widehat{w}_n(\xi)$ to derive a bound
$\widehat{w}_n(\xi)\le \tilde{G}_n(\tau_k)+ o(1)$ for some decreasing
one-parameter function $\tilde{G}_n$. 
Since linear independence is assumed for $P_{k-2}, P_{k-1}, P_k$ as well,
on the other hand we may apply Theorem~\ref{A} for index $k-1$ and 
combining we get 
\[
\widehat{w}_n(\xi) \le \min\{  \tau_k + 2n-3,  \tilde{G}_n(\tau_k) \} + o(1).
\]
We let $k\to\infty$ within our subsequence 
to make the error term disappear,
and the worst case scenario happens when the bounds coincide
which in particular means $\tau_k= \widehat{w}_n(\xi)-2n+3$. 
Plugging this into $\tilde{\Theta}_n$ yields
\[
\tilde{\Theta}_n(\widehat{w}_n(\xi), \widehat{w}_n(\xi)-2n+3,
\frac{n-1}{\widehat{w}_n(\xi)-n})\le 0
\]
and the claimed bound $\gamma_n$ follows after some rearrangements.

Now as in (ii) 
assume the linear independence hypothesis holds for all large $k$.
For the same reasons as in~\S~\ref{pthm1} we may again assume the sequence $(\tau_k)_{k\ge 1}$ is bounded from above. Then
choosing a subsequence of $k$ with $\tau_k\to \overline{\tau}<\infty$, by \eqref{eq:stein} and as $\tilde{\Theta}$ increases in $\tau_{\ell}$ we obtain
\begin{equation} \label{eq:hornstein}
\tilde{\Theta}_n= \tilde{\Theta}_n(\widehat{w}_n(\xi), \overline{\tau}, \overline{\tau})  \le \epsilon_1, \qquad k\ge k_1(\epsilon_1).
\end{equation}
This results in a bound of the form
$\widehat{w}_n(\xi)\le \tilde{H}_n(\overline{\tau})+ o(1)$ for the function 
\[
\tilde{H}_n(x)= x + n-1 + \frac{n-2}{x}.
\]
This function has a local minimum
at $x=\sqrt{n-2}$ and decreases
on $x\in (1,\sqrt{n-2})$, if non-empty, and increases for $x>\sqrt{n-2}$. Moreover it can be checked to be 
smaller than $x+2n-3$ from 
Theorem~\ref{A} for all $x>1$. By these observations
and since $\overline{\tau}\ge 1$, it follows from \eqref{eq:tauk} that
\[
\widehat{w}_n(\xi)\le \max\{ \tilde{H}_n(1) , \tilde{H}_n(\frac{n-1}{\widehat{w}_n(\xi)-n})  \} = \max\{ 2n-2, \tilde{H}_n(\frac{n-1}{\widehat{w}_n(\xi)-n}) \},
\]
where the left bound applies if $\overline{\tau}\le \sqrt{n-2}$ and the right
bound applies if $\overline{\tau} > \sqrt{n-2}$.
In case that the right hand side value is the maximum, rearrangements
lead to the bound
\[
\widehat{w}_n(\xi)\le \frac{ \sqrt{5}+1}{2}n - \frac{\sqrt{5}-1}{2}.
\]
Finally it is checked that precisely for integers 
$n\ge 4$ this is smaller than $2n-2$. 

\begin{remark}
	Assume a slight generalization of Lemma~\ref{lemur}, namely that the polynomials need not be irreducible of exact degree $n$ (which was implied by assumption \eqref{eq:2n} via Lemma~\ref{ichalt}) for the implication.
	Then the proof above shows that, for $n\ge 4$ we could improve the bound $\rho_n=2n-2$ in the setting (ii) if we can settle a non-trivial lower bound for $\overline{\tau}$. A related
	estimate for simultaneous approximation is obtained in~\cite{mamo}, however the result for linear forms is not implied immediately. 
	Thus for $n\ge 10$ we would improve the best known bounds $\alpha_n$ in \eqref{eq:bsch} from~\cite{buschlei}.
\end{remark}

%\begin{remark}
%	Combining~\cite{mamo} and \cite[Lemma~3]{ichmja}, 
%	we can find a non-trivial lower bound
%	for $\underline{\tau}$, which leads to a stronger bound in (ii) for %$n\ge 4$. However, the statement is rather technical so we omit it. 
%\end{remark}
%
%
%
%So again combining
%with Theorem~\ref{A} for index $k-1$, which applies by assumption, we %may let
%$\tau_k= \widehat{w}_n(\xi)-2n+3+o(1)$ again. Inserting into %\eqref{eq:hornstein}, and some rearrangements
%show that this is equivalent to
%\[
%(2n^2-5n+2)(\widehat{w}_n(\xi)-2n+2) \le \epsilon_2,
%\]
%which yields $\widehat{w}_n(\xi)\le 2n-2$ as soon as $n\ge 3$ as we let
%$k\to\infty$ and thus $\epsilon_2\to 0$.

\end{document}